\documentclass[a4paper,12pt]{amsart}
\usepackage{amssymb}
\usepackage{bbm}
\usepackage[dvips]{graphicx}
\usepackage{pdfsync}
\usepackage{color,xcolor}
\usepackage{verbatim}
\usepackage{enumitem}
\usepackage[margin=2cm]{geometry}

\newtheorem{thm}{Theorem}[section]
\newtheorem*{thm*}{Theorem}

\newtheorem*{lem*}{Lemma}

\newtheorem*{fct*}{Fact}

\newtheorem*{cor*}{Corollary}

\newtheorem*{prop*}{Proposition}
\theoremstyle{definition}
\newtheorem{defn}[thm]{Definition}
\newtheorem*{defn*}{Definition}
\theoremstyle{remark}
\newtheorem{rem}[thm]{Remark}
\newtheorem*{rem*}{Remark}

\newtheorem*{example*}{Example}

\newtheorem*{que*}{Question}

\newcommand{\abs}[1]{\left\vert#1\right\vert}
\newcommand{\set}[1]{\left\{#1\right\}}
\newcommand{\eps}{\varepsilon}

\newcommand{\CB}{\mathcal{B}}

\newcommand{\CN}{\mathcal{N}}
\newcommand{\CT}{\mathcal{T}}

\newcommand{\CS}{\mathcal{S}}

\renewcommand{\emptyset}{\varnothing}

\newcommand{\sd}{\bigtriangleup}

\usepackage{mathrsfs}

\renewcommand{\epsilon}{\varepsilon}

\renewcommand{\leq}{\leqslant}
\renewcommand{\geq}{\geqslant}

\title[Factoring group shift actions]{Factoring strongly irreducible group shift actions onto full shifts of lower entropy}

\author{Dawid Huczek, Sebastian Kopacz}

\address{\hskip- \parindent
Dawid Huczek, Faculty of Pure and Applied Mathematics, Wroclaw University of Science and Technology, Wybrze\.ze Wyspia\'nskiego 27, 50-370 Wroc\l aw, Poland}
\email{dawid.huczek@pwr.edu.pl}

\address{\hskip- \parindent
	Sebastian Kopacz, \emph{Faculty of Pure and Applied Mathematics, Wroclaw University of Science and Technology, Wybrze\.ze Wyspia\'nskiego 27, 50-370 Wroc\l aw, Poland}}
\email{sebastian.kopacz@pwr.edu.pl}

\subjclass[2010]{Primary 37B10; Secondary 37B40, 43A07}    

\keywords{countable amenable group, (dynamical) tiling, free action, topological entropy, comparison property}

\date{\today}

\begin{document}

\begin{abstract}
We show that if $G$ is a a countable amenable group with the comparison property, and $X$ is a strongly irreducible $G$-shift satisfying certain aperiodicity conditions, then $X$ factors onto the full $G$-shift over $N$ symbols, so long as the logarithm of $N$ is less than the topological entropy of $G$.
\end{abstract}
\maketitle

\section{Introduction}
A well-known result in the study of symbolic dynamical systems states that any 
subshift of finite type (SFT) with the action of $\mathbb{Z}$ and entropy 
greater or equal than $\log N$ factors onto the full shift over $N$ symbols -- 
this was proven in \cite{M} and \cite{B} for the cases of equal and unequal 
entropy respectively. Extending these results for actions of other groups has 
been difficult, and it is known that a factor map onto a full shift of equal 
entropy may not exist in this case (see \cite{BS}). Johnson and Madden showed 
in \cite{JM} that any SFT with the action of $\mathbb{Z}^d$, which has entropy greater 
than $\log N$ and satisfies an additional mixing condition (known as corner gluing), 
has an extension which is finite-to-one (hence of equal entropy) 
and maps onto the full shift over $N$ symbols. This result was later improved 
by Desai in \cite{D} to show that such a system factors directly onto the full 
shift, without the intermediate extension, and then by Boyle, Pavlov and Schraudner (\cite{BPS}) to replace the corner gluing by a weaker mixing condition (block gluing).

In this paper we use similar methods to adapt these constructions to symbolic dynamical systems with actions of amenable groups. Our approach requires three assumptions: that the group $G$ has the comparison property (satisfied, for instance, for all countable amenable groups with subexponential growth), that the system $X$ is strongly irreducible (which replaces the corner gluing condition, and allows the construction to be valid without assuming that the underlying system be a subshift of finite type), and that it has non-periodic blocks for all possible sets of periods (this condition seems reminiscent of faithfulness but we were not able to directly derive it from such an assumption). With these assumptions we prove that $X$ can be factored onto any full shift of smaller entropy (i.e. a full shift over $N$ symbols, where $\log N < h(X)$). We note that our method does not apply for the case of equal entropy ($\log N = h(X)$); as we mention earlier, Boyle and Schraudner have shown in \cite{BS} that there exist $\mathbb{Z}^d$-shifts of finite type which do not factor onto full shifts of equal entropy.

\section{Preliminaries}
In this section we establish the definitions, notation and standard facts we will use in this paper. Since this is mainly standard material, we omit most proofs and references. 
\subsection{Amenability, F\o lner sets and invariance}
Throughout this paper, $G$ we will denote a countable amenable group, and $(F_n)$ will denote a fixed F\o lner sequence, i.e. a sequence of finite subsets of $G$ such that for every $g\in G$ the sequence $\frac{\abs {gF_n\sd F_n}}{\abs{F_n}}$ tends to $0$ as $n$ goes to infinity. Multiplication involving sets will always be understood element-wise, so $gF_n$ is the set $\set{gf: f\in F_n}$, and $KF$ in the following definition denotes the set $\set{kf: k\in K, f\in F}$.
\begin{defn}
Let $K$, $F$ be finite subsets of $G$ and let $\eps>0$. We say that $F$ is $(K,\eps)$-invariant if $\abs{KF\sd F}<\eps\abs{F}$.
\end{defn}
The defining property of the F\o lner sequence can be equivalently (and usefully) restated thus: for every finite $K\subset G$ and every $\eps>0$ there exists an $N$ such that for every $n\geq N$ the set $F_n$ is $(K,\eps)$ invariant.
\begin{defn}
Let $D$, $F$ be finite subsets of $G$. Let $F_D=\set{g\in F: Dg\subset F}$. We will refer to $F_D$ as the $D$-interior of $F$.
\end{defn}
A straightforward computation shows that for any $\eps>0$, if $F$ is $(D,\frac{\eps}{\abs{D}})$-invariant, then $\abs{F_D}\geq (1-\eps)\abs{F}$ (we will describe such a relation in cardinalities by saying that $F_D$ is a \emph{$(1-\eps)$-subset of $F$}). Also observe that $F_{DK}=(F_D)_K$.
\subsection{Symbolic dynamical systems}
Let $\Lambda$ be a finite set (referred to as the \emph{alphabet}). The \emph{full $G$-shift} over $\Lambda$ is the product set $\Lambda^G$ with the product topology (induced by the discrete topology on $\Lambda$) endowed with the right-shift action of $G$:
\[ (gx)(h)=x(hg).\]
By a symbolic dynamical system, a shift space, or subshift, we understand any closed, shift-invariant subset $X$ of $\Lambda^G$.
\begin{defn}
For a finite $T\subset G$, by a block with domain $T$ we understand a function $B:T\to\Lambda$. If $X$ is a symbolic dynamical system over $\Lambda$, and $x\in X$, then we say that $B$ occurs in $x$ (at position $g$) if for every $t\in T$ we have $x(tg)=B(t)$, which we denote more concisely as $x(Tg)=B$. We say that $B$ occurs in $X$ if it occurs in some $x\in G$. 
\end{defn}
\begin{defn}
We say that a symbolic dynamical system $X$ is strongly irreducible (with irreducibility distance $D$) if for any blocks $B_1,B_2$ (with domains $T_1,T_2$) which occur in $X$, and any $g_1,g_2\in G$ such that $DT_1g_1\cap T_2g_2=\emptyset$, there exists an $x \in X$ such that $x(T_1g_1)=B_1$ and $x(T_2g_2)=B_2$.
\end{defn}
\begin{defn}
If $P$ is a finite subset of $G$, we say that a block $B$ with domain $T$ is $P$-aperiodic if for every $p\in P$ there exists a $t\in T$ such that $tp$ is also in $T$, and $B(tp)\neq B(t)$. We will say that a symbolic dynamical system $X$ is aperiodic if it has a $P-$aperiodic block for every finite $P\subset G$.
\end{defn}
\begin{defn}
For a fixed F\o lner sequence $(F_n)$, let $\CN_{F_n}(X)$ be the number of different blocks with domain $F_n$ which occur in $X$. The topological entropy of a symbolic dynamical system $X$ is defined as 
\[ h(X)=\lim_{n\to\infty}\frac{1}{\abs{F_n}}\log \CN_{F_n}(X).\]
\end{defn}
(In this paper we use logarithms with base 2, although as usual the theorems remain true if one defines entropy using any other base, so long as the choice remains consistent throughout). It is a standard fact that the obtained value of $h(X)$ does not depend on the choice of the F\o lner sequence. In fact, the relation between entropy and the number of blocks holds for any sufficiently invariant domain, as per the following theorem (\cite{LW})
\begin{thm}\label{thm:entropy_and_block_count}
    For any $\eps>0$ there exists an $N>0$ and $\delta>0$ such that if $T$ is an 
    $(F_n,\delta)$-invariant set for some $n>N$, then 
    $\CN_T(X)>2^{(h(X)-\eps)\abs{T}}$.
\end{thm}
\subsection{Quasitilings and tilings}
\begin{defn}
A \emph{quasitiling} of a countable amenable group $G$ is any collection $\CT$ of finite subsets of $G$ (referred to as \emph{tiles}) such that there exists a finite collection $\CS$ of finite subsets of $G$ (referred to as \emph{shapes}) such that every $T\in \CT$ has a unique representation $T=Sc$ for some $S\in\CS$ and some $c\in G$. We refer to such a $c$ as the \emph{center} of $T$. If the tiles of $\CT$ are disjoint and their union is all of $G$, then we refer to $\CT$ as a \emph{tiling}.
\end{defn}
\begin{rem}
If we enumerate the set of shapes of a quasitiling $\CT$ as $\set{S_1,S_2,\ldots,S_r}$, then we can identify $\CT$ with an element $x_{\CT}$ of the set $\set{0,1,\ldots,r}^G$ defined as $x_{\CT}(g)=j$ if $S_jg$ is a shape of $\CT$, and $x_{\CT}(g)=0$ otherwise. This in turn induces (via orbit closure) a subshift $X_{\CT}$, and any element of $X_{\CT}$ in turn corresponds to a quasitiling of $G$ which has the same disjointness, invariance and density properties as $\CT$. This allows us to discuss some properties of quasitilings using the notions of topological dynamics (in particular, it makes sense to talk of entropy and of factorizations), by interpreting these notions as applied to the corresponding subshifts.
\end{rem}
We will use several theorems which guarantee the existence of quasitilings and/or tilings satisfying certain properties. The first is proven in \cite{DHZ} and we will invoke it when constructing subsystems with specified entropy (in section \ref{sec:subsystem}) and marker blocks (in section \ref{sec:markers}).
\begin{thm}
For any $K\subset G$ and any $\eps>0$ there exists a tiling $\CT$ of $G$ such that all tiles of $\CT$ are $(K,\eps)$-invariant and $h(\CT)=0$.
\end{thm}

When we construct the factor map onto the full shift, we will rely on combining two other results. The first one originally appears in the seminal paper by Ornstein and Weiss (\cite{OW}), although in \cite{DHZ} it is stated and proven in an equivalent form closer to the one stated here  (the main difference being that the original version does not explicitly use the notion of lower Banach density):
\begin{thm}\label{thm:quasitiling}
For any $\delta>0$ and $N>0$ there exists a quasitiling $\CT$ of $G$ such that:
\begin{enumerate}
    \item the shapes of $\CT$ are all F\o lner sets, i.e. $\CS=\set{F_{n_1},F_{n_2},\ldots,F_{n_r}}$, where $N\leq n_1<n_2<\ldots <n_r$, and $r$ depends only on $\delta$.
    \item $\CT$ is $\delta$-disjoint, i.e. every tile $T\in \CT$ has a subset $T^\circ$ such that $\abs{T{\circ}}>(1-\delta)\abs{T}$ and these subsets are pairwise disjoint.
    \item $\CT$ is $(1-\delta)$-covering, i.e the lower Banach density of the union of $\CT$ is greater than $1-\delta$.
\end{enumerate}
\end{thm}

The last theorem we will need is proven in \cite{DZ}, although the authors do not state it as a stand-alone fact, but rather establish it as a step in proving Theorem 7.5. The cited paper also contains a definition and a lengthy discussion of the comparison property; here we will just recall that the class of groups with comparison property includes all amenable groups of subexponential growth and it remains an open question whether there are \emph{any} countable amenable groups without the comparison property.
\begin{thm}
Suppose $G$ is a countable amenable group with the comparison property, and $K$ is a finite subset of $G$ containing the neutral element $e$. For every $\eps>0$ there exist $\delta>0$ and $N>0$ such that every $(1-\delta)$-covering quasitiling $\CT$ whose tiles are pairwise disjoint and shapes are $(K,\delta)$-invariant and have cardinality larger than $N$ can be modified to a tiling $\CT^\circ$ whose shapes are $(K,\eps)$-invariant. Moreover, the sets of centres of $\CT$ and $\CT^\circ$ are identical, and there exists a finite $H\subset G$ such that for any $g\in G$ we can determine the tile of $\CT^\circ$ to which it belongs, so long as we know the set $Hg\cap \CT$ (i.e., using the language of topological dynamics, $\CT^\circ$ is a factor of $\CT$).
\end{thm}

We will combine these two results into the following form (which will be used in the main construction):
\begin{thm}\label{thm:tiling}
If $G$ is a countable amenable group with the comparison property, then for every $\eps>0$, $N>0$ and every finite $K\subset G$ there exists a quasitiling $\CT'$ of $G$ such that
\begin{enumerate}
\item the shapes of $\CT'$ are all F\o lner sets, i.e. $\CS=\set{F_{n_1},F_{n_2},\ldots,F_{n_r}}$, where $N\leq n_1<n_2<\ldots <n_r$, and $r$ depends only on $\eps$.
\item $\CT'$ factors onto an exact tiling $\CT''$ of $G$ whose shapes are $(K,\eps)$-invariant.
\end{enumerate}
\end{thm}
\section{Subsystem entropy\label{sec:subsystem}}

\begin{thm}\label{thm:subsystem_entropy}
If $X$ is a strongly irreducible subshift of positive entropy $h$, then the set of topological entropies of subshifts of $X$ is dense in $(0,h)$.
\end{thm}
\begin{proof}
Let $a$ and $b$ be such that $0<a<b<h$. We will show that there exists a subshift $Y\subset X$ such that $a\leq h(Y) \leq b$ (which is an equivalent formulation of the theorem). Fix a positive $\eta$ smaller than $(b-a)/2$ and note that for every sufficiently large $L$ there  exists a positive integer $n$ such that $\frac{1}{L}\log n$ is in the interval $(a+\eta,b-\eta)$. Combined with theorem \ref{thm:entropy_and_block_count} this lets us state the following: 
  \begin{fct*}
  There exists an $N>0$, $\delta>0$ and $M>0$ such that if $T$ is an 
    $(F_n,\delta)$-invariant set for some $n>N$, and $\abs{T}\ge M$, then $\CB_T(X)$ has a subset $\CB_T^*$ such that $a+\eta<\frac{1}{\abs{T}}\log\abs{\CB_T^*}<b-\eta$.
  \end{fct*}  
Let $D$ denote the irreducibility distance of $X$. By theorem \ref{thm:tiling} there exists a tiling $\CT$ of $G$ such that $h(\CT)=0$ and the shapes of $\CT$ can have arbitrarily good invariance properties, which we will specify within the next few sentences. Enumerate the shapes of $\CT$ by $S_1,S_2,\ldots,S_J$. There exist some $\delta^*,\delta^{**}$and $n$ such that if the shapes of $\CT$ are $(D,\delta^*)$ and $(F_n,\delta^{**})$-invariant, then for every $j$ we can choose a family of blocks $\CB_j\subset \CB_{(S_j)_D}(X)$ such that if we denote $N_j=\abs{\CB_j}$, then we have $a+\eta<\frac{1}{\abs{(S_j)_D}}\log N_j<b-\eta$. In addition, since for small enough $\delta^*$ the relative difference between $\abs{(S_j)_D}$ and $\abs{S_j}$ can be arbitrarily small, we can even write
\[a+\eta<\frac{1}{\abs{S_j}}\log N_j<b-\eta.\]

Now, let $Y$ be the orbit closure of the set of all points $x\in X$ such that for every $T\in \CT$ $x(T_D)\in \CB_j$, where $j$ is such that $S_j$ is the shape of $T$. Observe that strong irreducibility, combined with the fact that $T_D$ is disjoint from all other tiles of $\CT$, means (via a standard compactness argument) that we can choose the blocks $x(T_D)$ independently of each other, and any such choice will yield a valid element of $Y$. 

We will estimate the entropy of $Y$ by considering the number of blocks with domain $F_n$ as $n$ increases to infinity, beginning with estimating this number of blocks from above. Fix  some large $n$ and note that every block $B$ with domain $F_n$ that occurs in $Y$ has the property that there exists some right-translate $\CT'$ of $\CT$ such that for every tile $T$ of $\CT'$ such that $T\subset F_n$ the block $B(T_D)$ belongs to $\CB_j$, where $j$ is such that $S_j$ is the shape of $T$. Let $H_n$ be the number of ways in which the right-translates of $\CT$ can intersect $F_n$, or equivalently, the number of ways in which the right-translates of $F_n$ can intersect $\CT$. Since $\CT$ has entropy $0$, for large enough $n$ we have $\frac{1}{\abs{F_n}}\log H_n <\frac{\eta}{2}$.

For any such right-translate $\CT'$ let $l_j$ be the number of tiles of $\CT'$ with shape $S_j$ that are subsets of $F_n$. Note that if $n$ is large enough, then $\sum_{j=1}^Jl_j\abs{S_j}$ is almost equal to $\abs{F_n}$. Since for every tile of $\CT'$ the block $B(T_D)$ is a block from $\CB_j$, and thus one of $N_j$ possible blocks, the $D$-interiors of the tiles of $\CT'$ can be filled in at most $\left(\prod_{j=1}^JN_j^{l_j}\right)$ ways. We have no control over the symbols outside these interiors, but we know that there are at most $\abs{F_n}-\sum_{j=1}^Jl_j\abs{(S_j)_D}$ such symbols. It follows that the number of blocks associated with $\CT'$ is at most

\[ \left(\prod_{j=1}^JN_j^{l_j}\right) \cdot \abs{\Lambda}^{\abs{F_n}-\sum_{j=1}^Jl_j\abs{(S_j)_D}}.\]
Taking logarithms and dividing by $\abs{F_n}$, we obtain
\[ \begin{split}\frac{1}{\abs{F_n}}\sum_{j=1}^J l_j\log N_j + \frac{\abs{F_n}-\sum_{j=1}^Jl_j\abs{(S_j)_D}}{\abs{F_n}}\log\abs{\Lambda}\end{split}.\]
If $\delta^*$ was small enough so that the $D$-interiors of the tiles of $\CT'$ form a $(1-\frac{\eta}{2\abs{\Lambda}})$-covering quasitiling (which we can safely assume since $D$, $\eta$ and $\Lambda$ were known before we started the construction), then for large enough $n$ the second term will not exceed $\frac{\eta}{2}$. As for the first term, we can further estimate it as follows:
\[ \begin{split}
\frac{1}{\abs{F_n}}\sum_{j=1}^J l_j\log N_j = \frac{1}{\abs{F_n}}\sum_{j=1}^J l_j\abs{S_j}\frac{1}{\abs{S_j}}\log N_j <(\beta-\eta) \frac{1}{\abs{F_n}}\sum_{j=1}^J l_j\abs{S_j}<\beta-\eta.
\end{split}\]
It follows that for any right-translate $\CT'$ of $\CT$, the number of blocks in $Y$ with domain $F_n$ that have blocks from $\CB_j$ in every tile of $\CT'$ with shape $S_j$ does not exceed $2^{(b-\frac{\eta}{2})\abs{F_n}}$, and thus the cardinality of $\CB_{F_n}(Y)$ does not exceed $H_n2^{(b-\frac{\eta}{2})\abs{F_n}}$. This lets us estimate that
\[  \frac{1}{\abs{F_n}}\log \abs{\CB_{F_n}(Y)} \leq \frac{1}{\abs{F_n}}\log H_n+b-\frac{\eta}{2}<b,\]
and hence
\[ h(Y)<b.\]
The lower bound for entropy is much simpler: for any $n$ the number of blocks with domain $F_n$ which occur in $Y$ is equal to at least $\left(\prod_{j=1}^JN_j^{l_j}\right)$, where $l_j$ is the number of tiles of $\CT$ with shape $S_j$ which are subsets of $F_n$ (this time we do not even need to consider possible translates of $\CT$). Consequently,
\[ 
\frac{1}{\abs{F_n}}\log\abs{\CB_{F_n}(Y)}\geq \frac{1}{\abs{F_n}}\sum_{j=1}^J l_j\log N_j = \frac{1}{\abs{F_n}}\sum_{j=1}^J l_j\abs{S_j}\frac{1}{\abs{S_j}}\log N_j > (a+\eta)\frac{\sum_{j=1}^Jl_j\abs{S_j}}{\abs{F_n}},
\]
which for large enough $n$ will be greater than $a$, and therefore $h(Y)>a$, which concludes the proof.
\end{proof}

\section{Marker blocks}\label{sec:markers}
Our main objective in this section will be to prove the following theorem:
\begin{thm}\label{thm:marker}
Let $X$ be an aperiodic, strongly irreducible symbolic dynamical system (with the irreducibility distance $D$) over the alphabet $\Lambda$, with action of a countable amenable group $G$. Let $Y$ be a proper subshift of $X$. Then there exists a block $M$ with shape $K$, and a tiling $\CT$ (with a family of shapes $\CS$), satisfying the following conditions:
\begin{enumerate}
\item For every $T\in\CT$, any $t\in T$, and any $x\in X$, if $x(Kt)=M$, then the set $T_D\cap Kt$ is not empty, and the block $x(T_D\cap Kt)$ does not occur in $Y$.
\item For two different $g_1,g_2\in G$ and any $x\in X$, if $x(Kg_1)=x(Kg_2)=M$, then $Kg_2\setminus DKg_1\neq\emptyset$, and $x(Kg_2\setminus DKg_1)$ is a block which does not appear in $Y$. 
\end{enumerate}
\end{thm}
\begin{proof} For clarity, we will separate the proof into stages
	\begin{enumerate}
\item Since $Y$ is a proper subshift of $X$, there exists some block $B$ which occurs in $X$ but not in $Y$. Let $Z$ denote the shape of $B$. 
\item Let $P_e=Z^{-1}DZ$ and for any $g\in G$ lest $P_g=g^{-1}P_eg$. Observe that for any $g$, $P_g$ is a finite set (and all of these sets have equal cardinality), and $e\in P_g$ (because we assumed that $e\in D$). Let $P'_g$ be the largest (in terms of cardinality) subset of $P_g$ such that $P'_g\subset P_h$ for infinitely many $h$ (if there is more than one such subset of equal cardinality, we can choose any of them). Choose $g_0$ so that $P'_{g_0}$ has maximal cardinality among all the $P_g$'s. Set $P=P'_{g_0}$ and $G_P=\set{g\in G: P\subset P_g}$. Observe that if any $h\in G$ occurs in infinitely many $P_g$ for $g\in G_P$, then due to the maximality of $P$ we necessarily have $h\in P$. It follows that for any $E\subset G$ we have $E\cap P_g\subset P$ for all but finitely many $g\in G_P$.

\item By our assumption, there exists in $X$ a block $M_0$ with domain $K_0$ which is  $P$-aperiodic, i.e. for every $g\in P$ there exists a $k\in K_0$  such that $k$ and $kg$ are both elements of $K_0$, and $M_0(kg)\neq M_0(k)$. By strong irreducibility, we can also require that $M_0$ has $B$ as as subblock, and thus $M_0$ cannot occur in $Y$. 
\item Let $\CT'$ be a tiling of $G$, such that the shapes of $\CT'$ are supersets of $DZ$, let $S'$ denote the union of all shapes of $\CT'$, and let $\CT$ be another tiling, whose shapes are $(S',\delta)$-invariant, where $\delta$ is so small that for every shape $S$ of $\CT$, and every $s\in G$ the set $S_{D}s^{-1}$ includes an entire tile of $\CT'$. Let $S$ be the union of all shapes of $\CT$. By the invariance property established earlier, for any $s\in S\setminus S_{DK_0}$ the set $S_{D}s^{-1}$ includes an entire tile of $\CT'$, which allow us to choose a finite set $C'$ (consisting of centers of such tiles) such that for every $s\in S\setminus S_{DK_0}$ the set $S_{D}s^{-1}$ is a superset of $Zc'$ for some $c'\in C'$. 

Let $K_1=K_0\cup \bigcup_{c'\in C'}DZc'$. Strong irreducibility yields the existence of a block $M_1$ with domain $K_1$, such that $M_1(K_0)=M_0$, and for every $c'\in C'$ we have $M_1(Zc')=B$. Observe that for any tile $T\in \CT$, if $g\in T$, then either $K_0g\subset T_D$, or for at least one $c'\in C'$ we have $Zc'g\subset T_D$. Indeed, $T$ has the form $Sc$, where $S$ is one of the shapes of $\CT$. Let $s=gc^{-1}$. If $s\in S_{DK_0}$, then $DK_0s\subset S$, and thus $DK_0g=DK_0sc\subset Sc = T$. Otherwise, we know that for some $c'\in C'$ we have $Zc'\subset S_{D}s^{-1}$.  Therefore $Zc's\subset S_{D}$, which in turn yields $Zc'g=Zc'sc \subset S_{D}c=T_D$.


That way we have obtained a block $M_1$ with domain $K_1$ that satisfies the first property stated in our theorem. Note that any block that has $M_1$ as a subblock will retain this property.
\item Now choose $c_2$ as an element of $G_P$ which satisfies the following conditions (each of them is satisfied for all but finitely many elements of the group, and $G_P$ is infinite, which makes such a choice possible). The significance of these conditions is certainly not obvious at first, but will become clear later.
\begin{enumerate}
\item $c_2\notin Z^{-1}D^{-1}K_1$.\label{marker: irreducibility 1}
\item $c_2^{-1}\notin K_1^{-1}DK_1K_1^{-1}D^{-1}Z$\label{marker:1,2 overlap 1}.
\item $c_2^{-1}\notin K_1^{-1}DZZ^{-1}D^{-1}Z$\label{marker:1,2 overlap 2}.
\end{enumerate}
Also choose $c_3$ as an element of $G_P$ satisfying the following conditions (which is again possible because each of the following is true for all but finitely many elements of $G_P$)

\begin{enumerate}[resume]
\item $c_3\notin Z^{-1}D^{-1}(K_1\cup Zc_2)$.\label{marker: irreducibility 2}
\item $c_3\notin Z^{-1}D^{-1}K_1c_2^{-1}Z^{-1}DK_1$\label{marker:1 overlap 3, 2 overlap 1}.
\item $c_3\notin Z^{-1}D^{-1}K_1c_2^{-1}Z^{-1}DZc_2$\label{marker:1 overlap 3, 2 overlap 2}.
\item $c_3\notin  Z^{-1}D^{-1}Zc_2K_1^{-1}DK_1$\label{marker:2 overlap 3, 1 overlap 1}.
\item $c_3\notin  Z^{-1}D^{-1}Zc_2K_1^{-1}DZc_2$\label{marker:2 overlap 3, 1 overlap 2}.
    \item $P_{c_3}\cap P_{c_2}=P$\label{marker:1 overlap 1, 2 overlap 2, 3 overlap 3}.
\item   $c_3^{-1}\notin K_1^{-1}DK_1K_1^{-1}D^{-1}Z$\label{marker:1 overlap 1, 2 overlap 2, 3 overlap 1}.
\item $c_3^{-1}\notin c_2^{-1}Z^{-1}DZZ^{-1}D^{-1}Z$\label{marker:1 overlap 1, 2 overlap 2, 3 overlap 2}.
	\item $P_{c_3}\cap K_1^{-1}DZc_2=P$\label{marker:1 overlap 2, 2 overlap 1, 3 overlap 3}.
\item $c_3^{-1}\notin K_1^{-1}DZc_2K_1^{-1}D^{-1}Z$\label{marker:1 overlap 2, 2 overlap 1, 3 overlap 1}.
\item $c_3^{-1}\notin K_1^{-1}DZZ^{-1}D^{-1}Z$\label{marker:1 overlap 2, 2 overlap 1, 3 overlap 2}.
\end{enumerate}

Finally, we replace $K'$ with $K=K_1\cup Zc_2\cup Zc_3$, and we choose a block $M$ such that $M(K_1)=M_1$, $M(Zc_2)=M(Zc_3)=B$. Such an $M$ exists, because conditions (\ref{marker: irreducibility 1}) and (\ref{marker: irreducibility 2}) imply that $DZc_2$ is disjoint from $K_1$, and $DZc_3$ is disjoint from $K_1\cup Zc_2$, which allows us to invoke the strong irreducibility of $X$. Now, assume that for some $x\in X$ and $g\in G$ we have $x(Kg)=x(K)=M$. We will show that in such a situation, at least one of the sets $K_1g$, $Zc_2g$, $Zc_3g$ is disjoint with $DK$. This will require some laborious and repetitive computations, which we will separate into the following steps.
\begin{enumerate}
\item $x(Kg)=X(K)=M$ implies that for every $k\in K_1$ such that $kg$ is also in $K_1$, we have $M_1(k)=M_1(kg)$. Since $M_1$ is $P$-aperiodic, $g$ cannot belong to $P$. 
\item Observe that for any $g\in G$ each of the sets $DK_1$, $DZc_2$, $DZc_3$ is intersected by at most one of the sets $K_1g$ and $Zc_2g$. Indeed:
\begin{itemize}
\item  Suppose that $DK_1\cap K_1g$ and $DK_1\cap Zc_2g$ are both nonempty. This implies $g\in K_1^{-1}DK_1$ and $g\in c_2^{-1}Z^{-1}DK_1$, and thus
$K_1^{-1}DK_1\cap c_2^{-1}Z^{-1}DK_1\neq \emptyset$.
This, however would imply that $c_2^{-1}\in K_1^{-1}DK_1K_1^{-1}D^{-1}Z$, contradicting the condition (\ref{marker:1,2 overlap 1}). 
\item Similarly, if $DZc_2\cap K_1g$ and $DZc_2\cap Zc_2g$ are both nonempty, a similar reasoning tells us that the intersection  $K_1^{-1}DZc_2\cap c_2^{-1}Z^{-1}DZc_2$ is nonempty, thus so is $K_1^{-1}DZ\cap c_2^{-1}Z^{-1}DZ$, and therefore $c_2^{-1}\in K_1^{-1}DZZ^{-1}D^{-1}Z$, contradicting the condition (\ref{marker:1,2 overlap 2}). 
\item In the case where $DZc_3\cap K_1g$ and $DZc_3\cap Zc_2g$ are both nonempty, we similarly conclude that the set $K_1^{-1}DZc_3\cap  c_2^{-1}Z^{-1}DZc_3$ is nonempty, thus so is $K_1^{-1}DZ\cap  c_2^{-1}Z^{-1}DZ$, which is exactly the same conclusion as above (which means it also cannot hold).
\end{itemize}
\item As established above, even if $K_1g$ and $Zc_2g$ both intersect $DK$, they intersect different ``components'' of $DK$. We will now show that if either of $K_1g$ or $Zc_2g$ intersects $DZc_3$, then the other one is disjoint from $DK$. Indeed, suppose $K_1g\cap DZc_3$ is nonempty, which is equivalent to the condition $g\in K_1^{-1}DZc_3$.
\begin{itemize}
\item We have already established that in this scenario $Zc_2g$ cannot intersect $DZc_3$.
\item If $Zc_2g\cap DK_1$ is nonempty, then $g\in c_2^{-1}Z^{-1}DK_1$. Therefore, $c_2^{-1}Z^{-1}DK_1\cap K_1^{-1}DZc_3$ is nonempty, which gives us $c_3\in Z^{-1}D^{-1}K_1c_2^{-1}Z^{-1}DK_1$, contradicting the condition (\ref{marker:1 overlap 3, 2 overlap 1}).
\item If $Zc_2g\cap DZc_2$, then $g\in c_2^{-1}Z^{-1}DZc_2$. This means $c_2^{-1}Z^{-1}DZc_2\cap  K_1^{-1}DZc_3$ is nonempty, and thus $c_3\in Z^{-1}D^{-1}K_1c_2^{-1}Z^{-1}DZc_2$, contradicting (\ref{marker:1 overlap 3, 2 overlap 2})
\end{itemize}
On the other hand, if $Zc_2g\cap DZc_3$ is nonempty, then $g\in c_2^{-1}Z^{-1}DZc_3$, and we reason as follows:
\begin{itemize}
\item We already know $K_1g$ cannot intersect $DZc_3$.
\item If $K_1g\cap DK_1\neq\emptyset$, then $g\in K_1^{-1}DK_1$. Thus $c_2^{-1}Z^{-1}DZc_3\cap K_1^{-1}DK_1$ is nonempty, hence $c_3\in  Z^{-1}D^{-1}Zc_2K_1^{-1}DK_1$, contradicting (\ref{marker:2 overlap 3, 1 overlap 1})
\item If $K_1g\cap DZc_2$, then $g\in K_1^{-1}DZc_2$. Consequently, $ K_1^{-1}DZc_2\cap c_2^{-1}Z^{-1}DZc_3$ is nonempty, yielding $c_3\in  Z^{-1}D^{-1}Zc_2K_1^{-1}DZc_2$, and this is impossible by virtue of condition (\ref{marker:2 overlap 3, 1 overlap 2}).
\end{itemize}
\item There remain only two possible cases where $K_1g$ and $Zc_2g$ both overlap $DK$, and we will show that in both of those situations $Zc_3g$ must be disjoint from $DK$. The first possibility is that $K_1g\cap DK_1$ and $Zc_2g\cap DZc_2$ are both nonempty. This gives us $g\in K_1^{-1}DK_1$ and $g\in c_2^{-1}Z^{-1}DZc_2(=P_{c_2})$, and we need to again consider three possibilities
\begin{itemize}
\item Suppose $Zc_3g\cap DZc_3\neq\emptyset$, and thus $g\in c_3^{-1}Z^{-1}DZc_3=P_{c_3}$ this means that $g\in P_{c_2}\cap P_{c_3}$, but by virtue of (\ref{marker:1 overlap 1, 2 overlap 2, 3 overlap 3}) this means $g\in P$. However, we have already established in step $(a)$ that $g$ cannot be in $P$, so this situation is not possible.
\item If $Zc_3g\cap DK_1$ is nonempty, then $g\in c_3^{-1}Z^{-1}DK_1$. Combining that with the fact that $g\in K_1^{-1}DK_1$, we establish that $K_1^{-1}DK_1\cap c_3^{-1}Z^{-1}DK_1$ is nonempty, and thus $c_3^{-1}\in K_1^{-1}DK_1K_1^{-1}D^{-1}Z$, contradicting (\ref{marker:1 overlap 1, 2 overlap 2, 3 overlap 1}).
\item If $Zc_3g\cap DZc_2$ is nonempty, then $g\in c_3^{-1}Z^{-1}DZc_2$. Combining this with the fact that $g\in c_2^{-1}Z^{-1}DZc_2$, we see that $c_3^{-1}Z^{-1}DZc_2\cap c_2^{-1}Z^{-1}DZc_2\neq\emptyset$, and thus $c_3^{-1}Z^{-1}DZ\cap c_2^{-1}Z^{-1}DZ$ is also nonempty. This gives us $c_3^{-1}\in c_2^{-1}Z^{-1}DZZ^{-1}D^{-1}Z$, contradicting (\ref{marker:1 overlap 1, 2 overlap 2, 3 overlap 2}).
\end{itemize}
\item The final possibility to exclude is that $K_1g\cap DZc_2$ and $Zc_2g\cap DK_1$ are both nonempty. These conditions translate into $g\in K_1^{-1}DZc_2$ and $g\in c_2^{-1}Z^{-1}DK_1$. For the last time, we need to consider three possible cases for $Zc_3g$.
\begin{itemize}
\item If $Zc_3g\cap DZc_3\neq\emptyset$, then $g\in c_3^{-1}Z^{-1}DZc_3=P_{c_3}$. We also know that $g\in K_1^{-1}DZc_2$, but in light of condition (\ref{marker:1 overlap 2, 2 overlap 1, 3 overlap 3}) this means $g\in P$, which is not possible.
\item If $Zc_3g\cap DK_1\neq\emptyset$, we get $g\in c_3^{-1}Z^{-1}DK_1$. Combining this with $g\in K_1^{-1}DZc_2$, we have $c_3^{-1}Z^{-1}DK_1\cap K_1^{-1}DZc_2$, thus $c_3^{-1}\in K_1^{-1}DZc_2K_1^{-1}D^{-1}Z$, contradicting  (\ref{marker:1 overlap 2, 2 overlap 1, 3 overlap 1}).
\item If $Zc_3g\cap DZc_2\neq\emptyset$, it follows that $g\in c_3^{-1}Z^{-1}DZc_2$. We also know that  $g\in K_1^{-1}DZc_2$, so  $c_3^{-1}Z^{-1}DZc_2\cap K_1^{-1}DZc_2 \neq\emptyset$, and consequently $c_3^{-1}Z^{-1}DZ\cap K_1^{-1}DZ$ is also nonempty. This implies $c_3^{-1}\in K_1^{-1}DZZ^{-1}D^{-1}Z$, contradicting (\ref{marker:1 overlap 2, 2 overlap 1, 3 overlap 2}).
\end{itemize}
\end{enumerate}
\end{enumerate}
We have shown that if $x(Kg)=x(K)=M$, then $Kg\setminus DK$ is a superset of at least one out of $K_1g$, $Zc_2g$, and $Zc_3g$. Thus the block $x(Kg\setminus DK)$ has either $M_1$ or $B$ as a subblock, but neither of those blocks occur in $Y$, and therefore $x(Kg\setminus DK)$ also does not occur in $Y$. More generally, if $x(Kg_1)=x(Kg_2)=M$, then since $x(Kg_2)=x(kg_2g_1^{-1}g_1)=g_1x(Kg_2g_1^{-1})$, we can write $(g_1x)(K)=(g_1x)(Kg_2g_1^{-1})=M$. It follows that $g_1x(Kg_2g_1^{-1}\setminus DK)$ is a block which does not occur in $Y$, but $g_1x(Kg_2g_1^{-1}\setminus DK)=x(Kg_2\setminus DKg_1)$, so this block also does not occur in $Y$, as required.
\end{proof}
Note that the tiling $\CT$ constructed above can (and probably will) have shapes smaller than the domain of $M$. This will not be an obstacle in our construction, but nevertheless we note that one can replace $\CT$ it by a larger, congruent tiling (i.e one whose tiles are unions of tiles of $\CT$), the shapes of which can be arbitrarily large and have arbitrarily good invariance properties, and such a replacement we will retain the properties specified in the theorem. Thus we can make the following remark:
\begin{rem}
The tiling $\CT$ in Theorem \ref{thm:marker} can be chosen to be $(F,\eps)$ invariant for any $\eps>0$ and any finite $F\subset G$.
\end{rem}
\section{Constructing the extension}\label{sec:main}
\begin{thm}
If $G$ is a countable amenable group with the comparison property, and $X$ is an aperiodic, strongly irreducible symbolic dynamical system with the shift action of $G$, then for every $N\in \mathbb{N}$ such that $\log(N)<h(X)$ there exists a factor map from $X$ onto the full $G$-shift over $N$ symbols.
\end{thm}
\begin{proof}
By theorem \ref{thm:subsystem_entropy}, $X$ has a subsystem $Y$ such that $\log N < h(Y) < h(X)$. Before we delve into the technical details, here is a rough outline of the construction:
\begin{enumerate}
    \item The factor map will determine if an element of $X$ can be tiled (at least within some finite area around the neutral element of $G$) using a certain collection of large shapes.
    \item This decision will be made based on the occurrences of a marker block within a certain window.
    \item If such a local tiling can be found, there is a correspondence between blocks over its tiles and blocks in the full shift, which induces the image under the factor map (in other words, we define the map as a sliding block code). If the contents of $x$ does not induce a local tiling (which is entirely possible, and in fact more likely than not), the code just assigns $0$ to the image a the ``problematic'' coordinates.
    \item To show that this map is a surjection, we tile every element of the full shift using large shapes, and construct an element in $X$ that only has marker blocks at the centres of such shapes; and blocks from $Y$ elsewhere. Since $h(Y)>\log N$, the space not taken up by markers is enough to encode the entire element of the full shift.
\end{enumerate}

We can apply theorem \ref{thm:marker} to obtain a block $M$ with domain $K$, and tiling $\CT$ with the following properties:
\begin{enumerate}
    \item If $T$ is a tile of $\CT$ and for some $x\in X$ we have $x(Kt)=M$, then the (nonempty) block $x(T_D\cap Kt)$ does not occur in $Y$
    \item For any two different $g_1,g_2\in G$, if $x(Kg_1)=x(Kg_2)=M$, then the (nonempty) block $x(Kg_2\setminus DKg_1)$ does not occur in $Y$.
\end{enumerate}
Let $\eps=h(Y)-\log N$ and let $E$ denote the union of the shapes of $\CT$. There exists a $\delta$ such that if $S$ is any $(E,\delta)$-invariant set, then the union of tiles of $\CT$ which are contained in $S$ is a $(1-\frac{\eps}{2})$-subset of $S$.

We are about to apply theorem \ref{thm:tiling} (note: we are not yet applying this theorem; we are just discussing one of its parameters) with the parameter $\delta$, so we know it will yield a quasitiling with $r_\delta$ shapes $S_1,S_2,\ldots,S_{r_\delta}$, where $r_\delta$ will depends only on $\delta$. Let $K^*\subset G$ be a finite set such that $\abs{\CB_{K^*}(Y)}>r_\delta\cdot \abs{E}$, and thus there exists a surjective function $\psi: \CB_{K^*}(Y) \to \set{1,\ldots, r_\delta}\times E$. We can also assume that $K^*$ is disjoint from $DK$.


We can now apply theorem \ref{thm:tiling} to obtain a quasitiling $\CT'$ such that:
\begin{enumerate}
    \item $\CT'$ has $r_\delta$ shapes.
    \item $\CT'$ factors onto a tiling $\CT''$ with the same set of centres, and such that all shapes of $\CT''$ are $(E,\delta)$-invariant.
    \item Every shape $S$ of $\CT''$ is a superset of $D(K\cup K^*)$, and in fact the number of blocks with domain $S_D\setminus D(K\cup K^*)$ that occurs in $Y$ is greater than $N^{\abs{S}}$, hence there exists a surjective map $\Pi: \CB_{S_D\setminus D(K\cup K^*)}(Y)\to \set{0,1,\ldots,N-1}^S$.
\end{enumerate}
(The latter property easily follows from the fact that if the shapes of $\CT'$ are sufficiently large F\o lner sets, then $S_D\setminus D(K\cup K^*)$ can have arbitrarily large relative cardinality in $S$, and the entropy of $Y$ exceeds $\log N$).

We now have all the objects we need to define the factor map $\pi$ from $X$ onto the full $G$-shift over $N$ symbols. We will do so by describing a procedure to determine the symbol $\pi(x)(e)$ based on $x(H')$ for a certain finite $H'\subset G$.

Let $H'$ be large enough that the knowledge of how the set of centers of $\CT'$ intersects $H'$ allows us to determine the tile $T''$ of $\CT''$ such that $e\in T''$. Consider the set of all $c\in H'$ such that $x(Kc)=M$. For every such $c$, if we have $\psi(K^*c)=(j,g)$ (for $1\leq j\leq r_{\delta}$ and $g\in E)$, we obtain a set of the form $S_jgc$. There are now two possibilities:
\begin{itemize}
\item If the set of these shapes is equal to the intersection of $H'$ with some right-translate of $\CT'$, then it determines the intersection of $H$ with the same right-translate of $\CT''$, and in particular it uniquely assigns $e$ to a set of the form $Sh$ for some shape $S$ of $\CT''$ and some $h\in H$. If $x(S_Dh\setminus D(K\cup K^*)h)$ is some block $A$ which occurs in $Y$, then let $x(e)=\Pi(A)(h^{-1})$.
\item If uniquely determining $x(e)$ as above is not possible (or the resulting block $A$ does not occur in $Y$), set $\pi(x)(e)=0$.
\end{itemize}
The map defined above is a sliding block code (with window $H'$), and thus it is a factor map from $X$ onto some subset of the full $G$-shift over $N$ symbols, and the only non-trivial property left to verify is surjectivity. In other words, we need to show that for every $z\in \set{0,1,\ldots,N-1}^G$ there exists some $x\in X$ such that $\pi(x)=z$. For such an $z$, we will define its preimage $x$ by prescribing the content of $x$ within $D$-interiors of disjoint sets (mostly tiles of $\CT$); strong irreducibility means that such an $x$ will exist provided the individual blocks do occur in $X$. Enumerate the tiles of $\CT'$ as $T'_1,T'_2,\ldots$. For every $k$, the tile $T'_k$ has the form $S'_{j(k)}c'_k$, where $j(k)\in\set{1,2,\ldots,r_\delta}$. The center $c'_k$ belongs to some tile $S_kc_k$ of $\CT$, and thus $c'_k=g_kc_k$ for some $g_k\in E$. The set $K^*c_k$ is a subset of the $D$-interior of some union of finitely many tiles of $\CT$; denote this interior by $\widehat{K^*}c_k$. Let $x(Kc_k)=M$, let $x(((S_k)_D\setminus K)c_k)$ be any block from $Y$, and let $x(\widehat{K^*}c_k)$ be a block from $Y$ such that $\psi(x(K^*c_k))=(j(k),g_k)$. Let $\hat{S}_kc_k$ be the union of all tiles of $\CT$ that are disjoint from $D(Kc_k\cup \widehat{K^*}c_k)$ and contained within $(T'_k)_D$. There exists a block $A$ in $Y$, with domain $(S_k)_D\setminus D(K\cup K^*)$, such that $\Pi(A)=z(T'')$, where $T''$ is the tile of $\CT''$ whose center is $c_k$ (we know there is exactly one such tile), and we can extend $A$ to a block $A_k$ (also occurring with $Y$) with domain $\hat{S}_k$. Set $x(\hat{S}_kc_k)=A_k$. In addition, for any tile $T$ of $\CT$ which is not a subset of $(T'_k)_D$ for any $k$, we can set $x(T_D)$ to be any block in $Y$.

The above construction, together with the properties of $M$, means that $x(Kc)=M$ if and only if $c=c_k$ for some $k$ (this is because all $D$-interiors of tiles of $\CT$, except the places where we explicitly put the marker, were chosen to be blocks from $Y$). This means that for every $g$ we can uniquely determine the tile of $\CT''$ to which $g$ belongs, based on the contents of $x(H'g)$, and hence $\pi(x)=z$.
\end{proof}

%
%

\begin{thebibliography}{00}

\bibitem{B} M. Boyle. \emph{Lower entropy factors of sofic systems}, Ergodic Theory Dynam. Systems vol. 3, no. \textbf{4} (1983), 541–557
\bibitem{BPS} M.Boyle, R.Pavlov, M.Schraudner, \emph{Multidimensional sofic shifts without separation and their factors}, Transactions of the American Mathematical Society, Vol. 362, (2010), 4617-4653
\bibitem{BS} M. Boyle, M. Schraudner. \emph{$\mathbb{Z}^d$ shifts of finite type without equal entropy full shift factors}, Journal of Difference Equations and Applications \textbf{15} (2009), 47-52
\bibitem{D} A. Desai, \emph{A class of $\mathbb{Z}^d$ shifts of finite type which factors onto lowe entropy full shifts}, Proceedings of the American Mathematical Society, vol. 27, no. \textbf{8} (2009), 2613-2621
\bibitem{DHZ} T. Downarowicz, D. Huczek, G. Zhang. \emph{Tilings of amenable groups}, Journal f\"ur die reine und angewandte Mathematik (Crelles Journal), vol. 2019, no. \textbf{747} (2019), 277-298. \bibitem{DZ} T. Downarowicz, G. Zhang. \emph{Symbolic extensions of amenable group actions and the comparison property}. Preprint, 2019. arXiv:1901.01457 [math.DS]
\bibitem{JM} A. Johnson, K. Madden.
\emph{Factoring higher-dimensional shifts of finite type onto the
full shift}, Ergodic Theory and Dynamical Systems \textbf{25} (2005), 811-822
\bibitem{LW} E. Lindenstrauss, B. Weiss., \emph{Mean topological dimension}, Israel J. Math \textbf{115} (2000) 1-24.
\bibitem{M} B. Marcus. \emph{Factors and extensions of full shifts}, Monatsh. Math. \textbf{88} (1979), 239–247
\bibitem{OW} D. S. Ornstein, B.Weiss, \emph{Entropy and isomorphism theorems for actions of amenable groups}, J. Analyse Math. \textbf{48} (1987), 1–141.

\end{thebibliography}
%

\end{document}